\documentclass{amsart}
\usepackage{amsfonts,latexsym,amsthm,amssymb,amsmath,amscd,euscript,mathtools}
\usepackage{framed}
\usepackage{fullpage}
\usepackage{graphicx, xcolor}
\usepackage{gensymb}
\usepackage{mathrsfs}
\usepackage{enumerate}
\usepackage{tikz,pgfplots,tikz-cd}
\usetikzlibrary{arrows,positioning,calc,patterns}
\usepackage{amsmath,amssymb,amsfonts,amsthm}
\usepackage{listings}
\usepackage{courier}
\usepackage{enumerate}
\usepackage{mathtools}
\usepackage{hyperref}
\usepackage{tabularray}
\usepackage{color}

\allowdisplaybreaks[1]

\theoremstyle{definition}

\newtheorem{theorem}{Theorem}[section]
\newtheorem{proposition}[theorem]{Proposition}
\newtheorem{algorithm}[theorem]{Algorithm}
\newtheorem{lemma}[theorem]{Lemma}

\newtheorem{conjecture}[theorem]{Conjecture}

\newcommand{\Q}{\mathbb{Q}}

\newcommand{\Z}{\mathbb{Z}}
\newcommand{\C}{\mathbb{C}}

\renewcommand{\H}{\mathcal{H}}

\newcommand{\im}{\operatorname{im}}

\newcommand{\Gal}{\operatorname{Gal}}
\newcommand{\GL}{\operatorname{GL}}

\newcommand{\Pic}{\operatorname{Pic}}
\newcommand{\Jac}{\operatorname{Jac}}

\title{Intrinsic Subgroups and the $\ell$-adic Galois image}
\author{Jacob Greene}
\date{May 2026}

\begin{document}

\begin{abstract}
    Let $X$ be a geometrically irreducible smooth projective curve over a field $k$. Yamazaki et al.\ define a biadditive symmetric pairing $\langle -,-\rangle$ on the torsion subgroup of the Picard group $\Pic(X)$ with values in $k^\times \otimes \Q/\Z$. The \textit{intrinsic subgroup} $\Pic(X)_\mathrm{tors}^\mathrm{is}$ is the kernel of this pairing. When $X$ is an elliptic curve $E$, we can identify $E \simeq \Pic^0(E)$. We classify $E(k)_\mathrm{tors}^\mathrm{is}$ in purely algebraic terms for many elliptic curves over an arbitrary field $k$. We give a generalization of the analytic methods of Yamazaki et al.\ from $\Q$ to an arbitrary field $k \subset \C$. Lastly, for $k=\Q$, we describe an algorithm to explicitly compute $E(\Q)_\mathrm{tors}^\mathrm{is}$.
\end{abstract}

\maketitle

\section{Introduction}

Let $X$ be a geometrically irreducible smooth projective curve over a field $k$. In \cite{yyyy}, Yamazaki et al.\ construct a pairing
\[\langle-,-\rangle \colon \Pic(X)_\mathrm{tors} \times \Pic(X)_\mathrm{tors} \to k^\times \otimes \Q/\Z.\]
The pairing is symmetric and biadditive, and also satisfies the following:
\begin{enumerate}[1.]
    \item (Functoriality) Let $\phi \colon X \to Y$ be a map of schemes, $[D] \in \Pic(X)_\mathrm{tors}$ and $[E] \in \Pic(Y)_\mathrm{tors}$. Then
    \[\langle [D], \phi^\ast [E]\rangle_X = \langle \phi_\ast [D], [E]\rangle_Y.\]
    \item (Base change) If $k'/k$ is an extension of fields, $X'$ is the base change of $X$ to $k'$, and $\phi^\ast$ denotes the map induced by $k \hookrightarrow k'$, we have $\phi^\ast\langle [D],[E]\rangle = \langle \phi^\ast [D], \phi^\ast [E]\rangle'$ (where $\langle-,-\rangle'$ is the pairing on $\Pic(X')_\mathrm{tors}$).
\end{enumerate}

The kernel of the pairing is the \textit{intrinsic subgroup},
\[\Pic(X)_\mathrm{tors}^\mathrm{is} := \left\{[D] \in \Pic(X)_\mathrm{tors} : \langle [D],[E]\rangle = 0 \:\forall [E] \in \Pic(X)_\mathrm{tors} \right\}.\]
When $X$ is an elliptic curve $E$, we identify $E(k) \simeq \Pic^0(E)$, and so the pairing is on the torsion points of $E$ itself.

Yamazaki et al.\ consider the case where $E$ is an elliptic curve over $\Q$. Using an explicit analytic method, they arrive at a classification which can be summarized as follows:
\begin{theorem} \cite{yyyy} \label{thm:Q_classif}
    Let $E$ be an elliptic curve over $\Q$. Then
    \begin{enumerate}[i.]
        \item the intrinsic subgroup $E(\Q)_\mathrm{tors}^\mathrm{is} \subset E(\Q)_\mathrm{tors}$ is cyclic of order at most $5$;
        \item if $\#E(\Q)_\mathrm{tors} > 8$, then $E(\Q)_\mathrm{tors}^\mathrm{is}$ is trivial; and
        \item if $\#E(\Q)_\mathrm{tors} = 8$, then $\#E(\Q)_\mathrm{tors}^\mathrm{is} \neq 4$.
    \end{enumerate}
    The remaining isomorphism classes of $E(\Q)_\mathrm{tors}^\mathrm{is} \subset E(\Q)_\mathrm{tors}$ permitted under Mazur's theorem all appear infinitely often.
\end{theorem}

In Section~\ref{sec:algorithm}, we detail a procedure which, given the Weierstrass coefficients of an elliptic curve $E/\Q$ and the coordinates of a set of generators for $E(\Q)_\mathrm{tors}$, returns the order of $E(\Q)_\mathrm{tors}^\mathrm{is}$ and a generator for $E(\Q)_\mathrm{tors}^\mathrm{is}$. An implementation of this algorithm at \cite{code} was used to compute the intrinsic subgroup orders for every elliptic curve over $\Q$ cataloged in the LMFDB \cite{lmfdb}.

In Section~\ref{sec:general_fields}, we work over a general field $k$, and show that all elliptic curves admitting cyclic isogenies of a certain type must have intrinsic subgroup of a certain size. This section effectively generalizes the results of \cite{yyyy}. However, while they use analytic methods restricted to $\Q$ and perform casework on each isomorphism class of $E(\Q)_\mathrm{tors}^\mathrm{is}$, our algebraic methods allow us to work over a general field and consider $E(\Q)_\mathrm{tors}^\mathrm{is} \cong \Z/N\Z$ for arbitrary $N$.

Let $M,N \in \Z$ with $M \mid N$. Our results in Section~\ref{sec:general_fields} give a partial solution to the moduli problem of classifying elliptic curves over a general field $k$ having torsion $\Z/N\Z$ and intrinsic subgroup $\Z/M\Z$. Using a simple algebraic method, we give a sufficient (but not necessary) condition for $E$ to have intrinsic subgroup $\Z/M\Z$ in terms of isogenies admitted by $E$. As a consequence, we find that the non-cuspidal $k$-points of a particular modular curve, which we name $X_1(N)_M^{+1}$ (following the conventions of \cite{yyyy}), correspond to pairs $(E,P)$ where $E$ is an elliptic curve and $P \in E(k)_\mathrm{tors}$ such that $\langle P \rangle \cong \Z/N\Z$ and $(N/M)\langle P,P\rangle = 0$. In the case where $E(k)_\mathrm{tors} = \langle P\rangle$, we find $E(k)_\mathrm{tors}^\mathrm{is} \cong \Z/M\Z$. However, not every curve $E$ with the desired torsion and intrinsic subgroup arises as a rational point of the modular curve $X_1(N)_M^{+1}$. In Section~\ref{sec:complex}, we generalize the methods of \cite{yyyy} to show that, at least in the case where $k$ is a subfield of $\C$, every curve $E$ with $E(k)_\mathrm{tors} \cong \Z/N\Z$ and $E(k)_\mathrm{tors}^\mathrm{is} \cong \Z/M\Z$ arises as a rational point of some twist $X_1(N)_M^\zeta$ of the curve $X_1(N)_M^{+1}$ by a root of unity $\zeta$.

\section{Cyclic isogenies and the intrinsic subgroup}\label{sec:general_fields}

Purely using the functoriality of the pairing, we are able to give a sufficient condition in terms of an isogeny admitted by a curve to describe its intrinsic subgroup structure.

\begin{lemma}\label{lem:isog_ker}
    Let $E_1$, $E_2$ be elliptic curves over a field $k$, with an isogeny $\psi \colon E_1 \to E_2$ so that $\ker \psi = E_1(k)_\mathrm{tors}$.
    Let $\hat \psi$ be the dual isogeny of $\psi$. Then $\hat \psi(E_2(k)_\mathrm{tors}) \subset E_1(k)_\mathrm{tors}^\mathrm{is}$.
\end{lemma}

\begin{proof}
    The dual isogeny $\hat \psi \colon E_2 \to E_1$ equals the pullback $\psi^\ast \colon \Pic^0(E_2) \to \Pic^0(E_1)$ on $E_2 = \Pic^0(E_2)$ (it may be defined this way; see \cite{silverman}, III.6.1). Also, $\psi = \psi_\ast$ on $E_1 = \Pic^0(E_1)$. Let $Q \in E_2(k)_\mathrm{tors}$. Then for every $P \in E_1(k)_\mathrm{tors}$, it follows from the functoriality of $\langle-,-\rangle$ that
    \[\langle P, \hat \psi(Q)\rangle = \langle \psi(P),Q\rangle = \langle 0,Q\rangle = 0.\]
    Thus $\hat \psi(Q) \in E_1(k)_\mathrm{tors}^\mathrm{is}$.
\end{proof}

\begin{theorem}\label{thm:isog-is}
    Let $E$ be an elliptic curve over $k$ with $E(k)_\mathrm{tors} = \langle P \rangle$, where $P$ has order $N$. Suppose $E$ has a $k$-rational cyclic subgroup $C$ of order $MN$ containing $P$. Then $[N/M]P \in E(k)_\mathrm{tors}^\mathrm{is}$ and $M \mid\#E(k)_\mathrm{tors}^\mathrm{is}$.
\end{theorem}

\begin{proof}
    Choose $Q$ with $C = \langle Q \rangle$ and $MQ = P$. Let $\psi \colon E \to E_1$ be the isogeny with kernel $\langle P\rangle$. Then $\psi(Q)$ has order $M$. Notice that for every $\sigma \in \Gal(\bar k/k)$, we have $\sigma(Q) = (dN+1)Q$ for some $d \in \Z$ (because $\langle Q \rangle$ is defined over $k$ and $P \in E(k)$). Then, recalling $\psi$ is defined over $k$ also (so $\psi^\sigma = \psi$), we have
    \[\sigma(\psi(Q)) = \psi^\sigma(\sigma(Q)) = \psi(\sigma(Q)) = (dN+1)\psi(Q) = \psi(Q).\] Thus $\psi(Q)$ is Galois-stable, hence a $k$-rational point of $E_1$. Then, by Lemma~\ref{lem:isog_ker}, \[[N/M]P = [N]Q = \hat \psi(\psi(Q)) \in E(k)_\mathrm{tors}^\mathrm{is}.\]
    The order of $[N/M]P$ is $M$, so $M \mid\#E(k)_\mathrm{tors}^\mathrm{is}$.
\end{proof}

\section{Galois images and the intrinsic subgroup}\label{sec:complex}

In this section, let $k$ be a subfield of $\C$. We can phrase the hypotheses of Theorem~\ref{thm:isog-is} as a condition on the adelic Galois image of $E$. Suppose $E$ is an elliptic curve with adelic Galois image lying inside the subgroup
\[\Gamma :=\Gamma_{1,0}(N,MN) := \Gamma_1(N) \cap \Gamma_0(MN)= \left\{T: T \equiv \begin{pmatrix} 1 + N\ast & \ast \\ 0 & \ast \end{pmatrix} \bmod MN\right\} \leq \GL_2(\widehat \Z).\]
The modular curve $X_{1,0}(N,MN) := X_\Gamma$ parameterizes triples $(E,P,C)$ where $E$ is such an elliptic curve, $P$ is a point of $E$, and $C$ is a $k$-rational cyclic subgroup of order $MN$ containing $P$. Theorem~\ref{thm:isog-is} shows that all $(E,P,C)$ appearing as points of $X_{1,0}(N,MN)$ have intrinsic subgroup containing $\langle (N/M)P\rangle$. However, the converse is not true: the curve $X_{1,0}(N,MN)$ does not parameterize all elliptic curves with the given intrinsic subgroup. In general, as we will see, we need to consider several twists of $X_{1,0}(N,MN)$. We will show the following result.

\begin{theorem}\label{thm:twists}
    Let $M \mid N$ be integers. Let $k \subseteq \C$ be a field and $E$ an elliptic curve over $k$ with a point $P \in E(k)$ of order $N$ so that $(N/M)\langle P,P\rangle = 0$. Then there is some $\zeta \in \mu(k)$ such that $(E,P)$ corresponds to a $k$-rational point of a twist $X_1(N)_M^\zeta$ of the curve $X_{1,0}(N,MN)$.
\end{theorem}

The case $k=\Q$ was considered in \cite{yyyy}, where we see that if $E$ is an elliptic curve with cyclic torsion of order $N \in \{1,\dots,10,12\}$ and intrinsic subgroup of order $M \mid N$, then $E$ lies on either the modular curve $X_1(N)_M^{+}$, which is the same as our $X_{1,0}(N,MN)$, or the curve $X_1(N)_M^{-}$, which is a twist of $X_{1,0}(N,MN)$. They give an explicit planar model of $X_1(N)_M^\varepsilon$ in the form $f_N(t) = \varepsilon s^M$, where the $f_N$ are polynomials which they explicitly compute for each $N \in \{1,\dots,10,12\}$. Many aspects of their argument are applicable to a general field $k \subset\C$. We follow their argument broadly. In the place of the curves $X_1(N)_M^{\pm}$, we construct a family of curves $\{X_1(N)_M^\zeta : \zeta \in \mu(k)\}$, by giving completely explicit models. It will be evident from these models that they are all twists of $X_1(N)_M^{+1}$.

Let $E$ be an elliptic curve with a nonzero rational torsion point $P$ of order $N$. Through an affine transformation, we can put $(E,P)$ into \textit{Tate normal form}. That is, there exists an isomorphism $\phi \colon E \overset{\sim}{\to} E^{(N)}_{a,b}$ for some $a, b \in k$, with
\[E^{(N)}_{a,b} \colon \begin{cases}
    y^2 = x^3 + ax^2 + bx, \quad b(a^2 - 4b) \neq 0 & \text{if $N = 2$;} \\
    y^2 + axy + by = x^3, \quad b(a^3 - 27b) \neq 0 & \text{if $N = 3$;} \\
    y^2 + (1 + a)xy + by = x^3 + bx^2, \quad b \neq 0 & \text{ if $N \notin \{2,3\}$.}
\end{cases}\]
and $\phi(P) = (0,0)$. In the case $N \notin \{2,3\}$, the parameters $a,b \in k$ are uniquely determined by the ordered pair $(E,P)$, so $a,b = a(E,P), b(E,P)$.

If $N = 2$, we can ensure through an affine transformation $x \mapsto \lambda x$ that either $E \simeq E_{0,b}^{(2)} : y^2 = x^3 + bx$ or $E \simeq E_{1,b}^{(2)} \colon y^2 = x^3 + x^2 + bx$, i.e., that $\psi_2(a,b) := a^2 - a = 0$. Similarly, if $N = 3$, we can ensure by sending $y \mapsto \lambda y$ that either $a = 0$ or $a = 1$, so again $\psi_3(a,b) := a^2 - a = 0$.

For each $N > 3$, we can explicitly compute the coordinates $N(0,0) = (x_N/z_N,y_N/z_N)$ where $x_N,y_N,z_N \in k[a,b]$. Dividing out factors from $z_N$, we get a polynomial $\psi_N$ such that $(0,0)$ has order exactly $N$ on the curve $E_{a,b}^{(0)}$ if and only if $\psi_N(a,b) = 0$. Furthermore every such $(a,b) \in Z(\psi_N)$ defines an elliptic curve $E_{a,b}$ where $(0,0)$ has order $N$, except when this curve would be singular, i.e., when $\Delta(a,b) = 0$.

Then for all $N \geq 2$, the quasi-affine plane curve $Z(\psi_N) \setminus Z(\Delta)$ explicitly gives a plane model of the modular curve $Y_1(N)$. Then $\C(X_1(N)) \simeq \C(a,b)$.

Now, we can also choose a point $\tau \in \H$ so that $(E,P) \cong (\C/(\Z + \tau \Z), 1/N)$. Two points $\tau, \tau'$ identify the same pair $(E,P)$ if and only if they lie in the same $\Gamma_1(N)$-orbit; this is the standard construction of $X_1(N)$ as $(\Gamma_1(N))\backslash\H$. In \cite{yyyy}, it is shown that we can write $a=a(E,P)=a(\tau)$ and $b=b(\tau)$ as functions of $\tau$ in terms of the Jacobi theta function $\vartheta_1(z) = \vartheta_1(\tau;z)$,
\[a(\tau) = \frac{\vartheta_1(1/N)^4 \vartheta_1(4/N)}{\vartheta_1(2/N)^5}, \qquad b(\tau) = \frac{\vartheta_1(1/N)^5 \vartheta_1(3/N)^3}{\vartheta_1(2/N)^8}.\]
We have $\psi_N(a(\tau),b(\tau)) = 0$ for every elliptic curve $E_\tau$, hence as functions in $\C(X_1(N))$. So we can write $\C(X_1(N)) = \C(a(\tau),b(\tau))$.

With $E,P,\tau$ as above, it is shown in \cite{yyyy} that
\[-\langle P,P\rangle = f_N(\tau) \otimes 1/N \in k^\times \otimes \Q/\Z\]
where
\[f_N(\tau) = \left(\frac{\vartheta_1(2/N)^{2}}{\vartheta_1(1/N)\vartheta_1(3/N)}\right)^N.\]
We also get $f_N \in \C(X_1(N))$, and $f$ has $\Q$-rational Fourier coefficients; thus it can be written as a rational function $f_N(\tau) = f_N(a,b) \in \Q(a,b)$.

Now let $M \mid N$ and define
\[s(\tau) = \left(\frac{\vartheta_1(2/N)^{2}}{\vartheta_1(1/N)\vartheta_1(3/N)}\right)^{N/M}.\]
Now $s(\tau)$ is not in general a $\Gamma_1(N)$-modular function, but in \cite{yyyy}, it is checked that $s$ is $\Gamma_{1,0}(N,MN)$-modular. So $s \in \C(X_{1,0}(N,MN))$. Then $\C(s,a,b) \subseteq \C(X_{1,0}(N,MN))$. Also $s^M \in \C(X_1(N))$, and so $[\C(s,a,b):\C(a,b)] = M$. But the map $X_{1,0}(N,MN) \to X_1(N)$ has degree $M$. Thus $\C(s,a,b) = \C(X_{1,0}(N,MN))$, and the two equations $z_N(a,b) = 0$, $f_N(a,b) = s^M$ define $X_{1,0}(N,MN)$ over $k$.

Now we can prove Theorem~\ref{thm:twists}.
\begin{proof}[Proof (Theorem~\ref{thm:twists}).]
    Suppose that $E$ has the point $P$ of order $N$ such that $[N/M]\langle P,P\rangle = 0$. Choose $\tau$ with $(\C/(\Z + \tau \Z),1/N) \simeq (E,P)$. Write $a:=a(\tau)$, $b:=b(\tau)$. Then we have $f_N(a,b) \otimes 1/M = 0$, i.e., $f_N(a,b) = \zeta s^M$ for some root of unity $\zeta$ and $s \in k$. Define the curve $X_1(N)_M^\zeta$ explicitly by the equations $\psi_N(a,b) = 0$ and $f_N(a,b) = \zeta s^M$; we see that $(E,P)$ corresponds to a $k$-point of $X_1(N)_M^\zeta$. Furthermore, $X_1(N)_M^{+1} \cong X_{1,0}(N,MN)$, and each $X_1(N)_M^\zeta$ is a twist of $X_1(N)_M^{+1}$. 
\end{proof}

Notice that not \textit{every} elliptic curve parameterized by the curve $X_1(N)_M^\zeta$ has $E(k)_\mathrm{tors}^\mathrm{is} \simeq \Z/M\Z$; namely, if $E(k)_\mathrm{tors}$ is strictly larger than $\langle P\rangle$, then $E(k)_\mathrm{tors}^\mathrm{is}$ may be smaller than $\langle (N/M)P\rangle$.

Write $\Gamma := \Gamma_0(MN) \cap \Gamma_1(N)$. Another way to view this result is a condition on the Galois image of $E$. We showed in Section~\ref{sec:general_fields} that when the Galois image of $E$ is $\Gamma$, then $E(k)_\mathrm{tors}^\mathrm{is} \simeq \Z/M\Z$; we can now see that the converse is true up to conjugacy.

\begin{theorem}
    Let $M, N \in \Z$ and $M \mid N$, and suppose $k \subseteq \C$. Let $E$ be an elliptic curve over $k$, and suppose $E(k)_\mathrm{tors} \simeq \Z/N\Z$ and $E(k)_\mathrm{tors}^\mathrm{is} \simeq \Z/M\Z$. Write $H \leq \GL_2(\widehat \Z)$ for the Galois image of $E$. Then $H$ is conjugate to a subgroup of $\Gamma := \Gamma_{1,0}(N,MN)$.
\end{theorem}

\begin{proof}
    Twists of $X_\Gamma$ are modular curves of the form $X_H$, where $H$ is conjugate to $\Gamma$. By Theorem~\ref{thm:twists}, $E$ lies on a twist of $X_\Gamma$.
\end{proof}

\section{Computing Intrinsic Subgroups Over $\Q$.}\label{sec:algorithm}

In this section, we outline an algorithm to give $E(\Q)_\mathrm{tors}^\mathrm{is}$, in the case where $E(\Q)_\mathrm{tors}$ is cyclic of order at least $4$. Along with Theorem~\ref{thm:Q_classif}, which gives the classification of possibilities for $E(\Q)_\mathrm{tors}^\mathrm{is}$, we use the following result, which is involved in the proof of that classification:

\begin{proposition} \cite{yyyy} \label{prop:fnstable}
    \item Let the elliptic curve $E/\Q$ be given by $E \colon y^2 + (1+a)xy + by = x^3 + bx^2$. If the point $P=(0,0)$ has order $N \geq 4$, then the pairing $-\langle P,P\rangle = f_N(t(a,b)) \otimes 1/N$, where the rational functions $f_N$, $t$ are given by the following table:
    \begin{center}
        \begin{tabular}{r|l|l}
            $N$ & $f_N(t)$ & $t(a,b)$ \\ \hline
            $4$ & $t$ & $b$ \\
            $5$ & $t$ & $b$ \\
            $6$ & $t(1-t)^2$ & $1-b/a$ \\
            $7$ & $t(1-t)^4$ & $1-b/a$ \\
            $8$ & $t(1-t)^2(1+t)^4$ & $a/b-1$ \\
            $9$ & $t(1-t)^4(1-t+t^2)^3$ & $(a-a^2-b)/(a-b)$ \\
            $10$ & $t(1-t)^2(1+t)^8(1+t-t^2)^5$ & $(a-a^2-b)/a^2$ \\
            $12$ & $t(1-t)^2(1-t+t^2)^3(1+t^2)^4(1+t)^6$ & $-(a^3-ab+b^2)/(a-b)^2$
        \end{tabular}
    \end{center}
\end{proposition}

The function $f_N(t(a,b))$ is the same as the function $f_N(a,b)$ which we defined in a more general context in Section~\ref{sec:complex}. The existence of the single parameter $t$ is due to the fact that all of the curves we are working with have genus $0$.

The proposition yields a straightforward procedure to compute $\langle P,P\rangle$ given $(E,P)$, and thereby to compute $E(\Q)_\mathrm{tors}^\mathrm{is}$. 

\begin{algorithm} Given an elliptic curve $E/\Q$ with Weierstrass coefficients $a_1,a_2,a_3,a_4,a_6$ and a generator $P=(x_0,y_0)$ for $E(\Q)_\mathrm{tors}$, assuming $P$ has order at least $4$, compute $(M,Q)$, where $E(\Q)_\mathrm{tors}^\mathrm{is} = \langle Q\rangle$ has order $M$.
\begin{enumerate}[1.]
    \item Find the order $N$ of $P$.
    \item Put $E$ into Tate normal form through the following transformations:
    \begin{enumerate}[i.]
        \item Send $(x,y) \mapsto (x-x_0,y-y_0)$, to set $a_6 = 0$;
        \item Send $(x,y) \mapsto (x,y+\frac {a_4}{a_3} x)$, to put $E$ into the form $E \colon y^2 + b_1 xy + b_3 y = x^3 + b_2x^2$.
        \item Send $(x,y) \mapsto (\lambda^2 x, \lambda^3 y)$ where $\lambda = b_3/b_2$, to put $E$ into the form $E \colon y^2 + (1+a)xy + by = x^3 + bx^2$. 
    \end{enumerate}
    Note that $a_3 = 0$ iff $P$ is $2$-torsion and $b_2 = 0$ iff $P$ is $3$-torsion. So by our assumption on $P$, we may divide by them.
    \item Using the functions in the table from Proposition~\ref{prop:fnstable}, compute $f_N(t(a,b))$, with $a$ and $b$ the Tate normal form coefficients.
    \item Find the highest exponent $e$ for which $|f_N(t(a,b))| = r^e$ for some $r \in \Q$, or set $e = 0$ if $f_N(t(a,b)) = 1$.
    \item Let $M = \gcd\{e,N\}$ and $Q = (N/M)P$. Return $(M,Q)$.
\end{enumerate}
\end{algorithm}

To see that the output is correct, notice that
\[Q \in E(\Q)_\mathrm{tors}^\mathrm{is} \iff \langle P,Q \rangle = 0 \iff \exists \mu \in \mu(\Q) = \{\pm 1\}, s \in \Q,\: f_N(t(a,b)) = \mu s^M.\]
A similar, but more involved, procedure allows us to compute the intrinsic subgroup in the cases $N = 2,3$, and when $E(\Q)_\mathrm{tors}$ is not cyclic. The full algorithm is implemented in \cite{code}.

\section{Discussion}\label{sec:end}

We found that the curve $X := X_{1,0}(N,MN)$ gives a family of curves having (with a few exceptions) $E(k)_\mathrm{tors} \simeq \Z/N\Z$ and $E(k)_\mathrm{tors}^\mathrm{is} \simeq \Z/M\Z$. We also showed that, when $k \subset \C$, every such curve lies on one of several twists of $X$. It seems natural to expect that the same curves also solve the moduli problem over a general field; however, our methods are not sufficient to show this. That is, the following seems likely (at least if $\operatorname{char}(k) \nmid N$):

\begin{conjecture}
    Let $k$ be any field and $E$ an elliptic curve over $k$ with $E(k)_\mathrm{tors} = \langle P\rangle \simeq \Z/N\Z$. Then $(N/M)P \in E(k)_\mathrm{tors}^\mathrm{is}$ iff $\im \rho_E \leq H$, where $X_H$ is one of the twists of $X_{1,0}(N,MN)$ appearing in Theorem~\ref{thm:twists}. That is, the moduli problem of elliptic curves having torsion $\Z/N\Z$ and intrinsic subgroup $\Z/M\Z$ is solved over a general field by the curves $X_1(N)_M^\zeta$.
\end{conjecture}

For simplicity, we restrict our attention in Section~\ref{sec:general_fields} to elliptic curves. However, it seems likely that our results can be phrased for a general curve $X$, in terms of isogenies admitted by the Jacobian $J=\Jac(X)$. In particular, Lemma~\ref{lem:isog_ker} could be proven similarly. This is exactly what we showed for subfields of $\C$ in Section~\ref{sec:complex}.


Our algorithm in Section~\ref{sec:algorithm} uses very little that is specific to $\Q$; most of the steps would work just as well using the rational functions $f_N(a,b)$, which have the necessary properties over any field. The largest source of additional complexity is that if we vary the field, $N$ is no longer uniformly bounded; we can no longer precompute the $f_N$. The functions $f_N$ can still be computed using their Fourier expansions, but the degree grows linearly with $N$. The other step which does not generalize immediately is identity testing in $k^\times \otimes \Q/\Z$: given $z \in k^\times$ and $n \in \mathbb \Z_{>0}$, we must determine whether $z \otimes \frac 1n = 1$ in $k^\times \otimes \Q/\Z$, or equivalently whether $z \in \mu(k) \cdot (k^\times)^n$. This can be done quickly for many classes of fields, including global fields and local fields; however, for other $k$ it may not be easy.


Throughout, we focused on the case with $E(k)_\mathrm{tors}$ cyclic. The analytic approach used in Section~\ref{sec:complex} and in \cite{yyyy} does not immediately generalize to the noncyclic case. However, it seems reasonable to expect that there is still a relationship between $E(k)_\mathrm{tors}^\mathrm{is}$ and the Galois image of $E$, even when $E(k)_\mathrm{tors}$ is noncyclic. What is this relationship? Are there cases in which $E(k)_\mathrm{tors}$ may not be cyclic, but $E(k)_\mathrm{tors}^\mathrm{is}$ admits a description purely in terms of the isogenies of $E$, analogous to Theorem~\ref{thm:isog-is}?

Some other basic questions also remain unanswered. In general, given a field $k$, which groups can appear as $E(k)_\mathrm{tors}^\mathrm{is}$? Is it true that, given any $G = \Z/N_1\Z \oplus \Z/N_2\Z$ and $H \leq G$, there exists a field $k$ and an elliptic curve $E/k$ with $E(k)_\mathrm{tors} \simeq G$ and $E(k)_\mathrm{tors}^\mathrm{is} \simeq H$? When there exists such a curve, are there infinitely many such curves? For example, in the case $G = \Z/N\Z$ and $H = \Z/M\Z$, Faltings' theorem implies that, whenever the genus of $X_{1,0}(N,MN)$ is at least $2$, there are finitely many such curves over any fixed number field.


\section*{Acknowldgements}

The author thanks Prof. Andrew Sutherland for his continuous support in this exploration.

\end{document}